\documentclass[11pt]{article}

\usepackage{fullpage}
\usepackage{amsmath}
\usepackage{amsthm}
\usepackage{amssymb}
\usepackage{amsfonts}
\usepackage{subfig}
\usepackage{multirow}
\usepackage{verbatim}
\usepackage{enumerate}
\usepackage{array}

\newcommand{\todo}[2][yo]{}

\usepackage{tikz}
\tikzstyle{vertex}=[circle,draw=black,fill=black,inner sep=0mm,minimum size=1mm]
\tikzstyle{ar}=[<=1pt,>=stealth',semithick]

\usepackage{plevel}
\usepackage{TikzFigures}

\usepackage{graphicx}
\graphicspath{%
    {converted_graphics/}
    {/}
    {Images/}
}
\begin{document}

\title{Conjugacy and Iteration of Standard Interval Rank in \\ Finite Ordered Sets\thanks{ PNNL-SA-105143.}}
\date{}

\author{
	Cliff Joslyn\thanks{National Security Directorate, Pacific
Northwest National Laboratory. Corresponding author, cliff.joslyn@pnnl.gov,
206-552-0351.},
	Emilie Hogan\thanks{Fundamental Sciences Directorate, Pacific
Northwest National Laboratory}, and
	Alex Pogel\thanks{Physical Science Laboratory, New Mexico State University}
	}

\maketitle

\begin{abstract}

In order theory, a rank function measures the vertical ``level'' of a
poset element. It is an integer-valued function on a poset which increments
with the covering relation, and is
only available on a graded poset. Defining a vertical measure to an
arbitrary finite poset can be accomplished by extending a rank
function to be interval-valued \cite{JoCHoE14}. This establishes
an order homomorphism from a base poset to a poset over real
intervals, and a standard (canonical) specific interval rank function
is available as an extreme case. Various ordering relations are available over
intervals, and we begin in this paper by considering conjugate orders
which ``partition'' the space of pairwise comparisons of order
elements. For us, these elements are real intervals, and we consider
the weak and subset interval orders as (near) conjugates. It is also
natural to ask about interval rank functions applied reflexively on whatever
poset of intervals we have chosen, and thereby a general iterative
strategy for interval ranks. We explore the convergence properties of
standard and conjugate interval ranks, and conclude with a discussion
of the experimental mathematics needed to support this work.

\end{abstract}

\tableofcontents

\section{Introduction}


A characteristic of partial orders and partially ordered sets (posets) as used in
order theory \cite{DaBPrH90} is that they are
amongst the simplest structures which can be called ``hierarchical'' in the sense of
admitting to descriptions in terms of {\em levels}.
These levels are typically identified via a rank
function as an integer-valued function on a poset which increments
with the covering relation. But rank
functions are only available on graded posets, where all saturated
chains connecting two elements are the same length. Effectively, each
element in a graded poset can be assigned a unique, singular vertical
level.

We have extended the concept of a rank function to all finite posets
by making it interval-valued \cite{JoCHoE14}, and establishing it as
an order homomorphism from a base poset to a poset over real
intervals, for one of many possible order relations on intervals. An
interval-valued rank reflects information about chain lengths to an
element dually from both the top and the bottom of the poset, and a
standard, or canonical, interval rank function is available as the
largest such function, measuring specifically the maximum length
maximum chain from both top and bottom.

As mentioned above, various ordering relations are available over
intervals. The order ordinarily identified on intervals we call the
strong order, where comparability is equivalent to disjointness. But
we have identified \cite{JoCHoE14} two other orders as being much more
useful for interval rank function available on intervals: a weak order
(actually the product order on endpoints) and ordering by
subset. These are (nearly) conjugate orders, in that they
``partition'' the space of pairwise comparisons of order elements. We
begin this paper by considering standard interval rank for the
conjugate to the weak order.  And since interval rank functions are
order homomorphisms to an interval order, it is also natural to ask
about interval rank functions applied reflexively on whatever poset of
intervals we have chosen. This motivates studying the structure of the
homomorphic image of a poset under the standard interval rank
operator, both singly and in repeated iterations.

We close with a discussion of how the use of experimental mathematics, using
computers as a tool to explore problems and form conjectures, influenced this
work.

\section{Preliminaries}

Throughout this paper we will use $\N$ to denote  $\{ 0, 1, \ldots \}$, the
set of integers greater than or equal to 0, and for $N \in \N$, the set of
integers between 0 and $N$ will be denoted $\N_N \define \{ 0, 1, \ldots, N
\}$.

\subsection{Ordered Sets}\label{OrderedSetsDefns}

See e.g.\ \cite{DaBPrH90,ScB03,TrW92} for the basics of order theory, the
following is primarily for notational purposes.

Let $P$ be a finite set of elements with $|P| \ge 2$, and $\le$
be a binary relation on $P$ (a subset of $P^2$) which is reflexive,
transitive, and antisymmetric. Then $\le$ is a {\bf partial order}, and the
structure $\poset = \tup{ P, \le }$ is a {\bf partially ordered set} (poset).
Denote $a < b$ to mean that $a \le b$ and $a \neq b$. $<$ is its own binary
relation, a {\bf strict order} $<$ on $P$, which is an irreflexive partial
order. A strict order $<$ can be turned into a partial order through {\bf
reflexive closure}: $\le
\define < \un \,\{ \tup{ a, a } \st a \in P \}$.

For any pair of elements $a,b \in P$, we say $a \le b \in P$ to mean that $a,b
\in P$ and $a \le b$. And for $a \in P, Q \sub P$, we say $a \le Q$ to mean
that $\forall b \in Q, a \le b$. If $a \le b \in P$ or $b \le a \in P$ then
we say that $a$ and $b$ are {\bf comparable}, denoted $a \sim b$. If not,
then they are {\bf incomparable}, denoted $a \| b$. For $a,b \in P$, let $a
\cover b$ be the {\bf covering relation} where $a \le b$ and $\nexists c \in
P$ with $a < c < b$.

A set of elements $C \sub P$ is a {\bf chain} if $\forall a, b \in C, a \sim b$.
If $P$ is a chain, then $\poset$ is called a {\bf total order}. A chain $C
\sub P$ is {\bf maximal} if there is no other chain $C' \sub P$ with $C \sub
C'$. Naturally all maximal chains are {\bf saturated}, meaning that $C = \{
a_i \}_{i=1}^{|C|} \sub P$ can be sorted by $\le$ and written as $C = a_1
\cover a_2 \cover \ldots \cover a_{|C|}$. The {\bf height} $\height(\poset)$
of a poset is the size of its largest chain. Below we will use $\height$
alone for $\height(\poset)$ when clear from context.

A set of elements $A \sub P$ is an {\bf antichain} if $\forall a,b \in A, a \|
b$. The {\bf width} ${\cal W}(\poset)$ of a poset is the size of its largest
antichain.

A partial order $\le$ generates a unique {\bf dual} partial order $\ge$,
where $a \ge b \in P$ iff $b \le a \in P$. Given $\po = \tup{P, \leq}$ we
denote by $\po^* = \tup{P,\geq}$ the dual of $\po$. Two partial orders
$\leq_1, \leq_2$ on the same set $P$ are said to be {\bf conjugate} if every
pair of distinct elements of $P$ are comparable in exactly one of $\leq_1$ or
$\leq_2$, i.e.\ that:
	\[ \forall a \neq b \in P,	\quad
		( a \sim_1 b \h{ and } a \|_2 b ) \h{ or }
		( a \|_1 b  \h{ and } a \sim_2 b ).	\]

For any subset of elements $Q \sub P$, let $\psub{Q} = \tup{ Q, \le_Q }$ be the
sub-poset determined by $Q$, so that for $a,b \in Q$, $a \le_{Q} b \in Q$ if
$a \le b \in P$.

For any element $a \in P$, define the {\bf up-set} or {\bf principal filter}
$\up a \define \{ b \in P \st b \ge a \}$, {\bf down-set} or {\bf principal
ideal} $\down a \define \{ b \in P \st b \le a \}$, and {\bf hourglass}
$\Xi(a) \define \up a \un \down a$. For $a \le b \in P$, define the {\bf
interval} $[a,b] \define \{ c \in P \st a \le c \le b \} = \up a \int \down b
$.

For any subset of elements $Q \sub P$, define its maximal and minimal elements
as
	\[ \Maxxx(Q) \define
		\{ a \in Q \st {\nexists} b \in Q, a < b \}	\sub Q  \]
	\[ \Minnn(Q) \define
		\{ a \in Q \st {\nexists} b \in Q, b < a \} \sub Q,	\]
	called the {\bf roots} and {\bf leaves} respectively. Except where
noted, in this paper we will assume that our posets $\poset$ are bounded, so
that $\bot \le \top \in P$ with $\Maxxx(P) = \{ \top \}, \Minnn(P) = \{ \bot
\}$. Since we've disallowed the degenerate case of $|P|=1$, we have $\bot <
\top \in P$. All intervals $[a,b]$ are bounded sub-posets, and since $\poset$
is bounded, $\forall a \in P, \up a = [ a, \top ], \down a = [ \bot, a ]$,
and thus $\bot,\top \in \Xi(a) \sub \poset = [\bot,\top]$. An example of a
bounded poset and a sub-poset expressed as an hourglass is shown in
\fig{allex2}.

\mytikz{\bddExWithHourglass}{>=latex, line width=0.75pt}{allex2}{(Left) The
Hasse diagram (canonical visual representation of the covering relation
$\cover$) of an example bounded poset $\poset$. (Right) The Hasse diagram of
the sub-poset $\psub{\Xi(J)}$ for the hourglass $\Xi(J) = \up J \un \down J =
[ J, \top ] \un [ \bot, J ] = \{ J,C,K,\top \} \un \{\bot,J\} = \{ \bot, C,
J, K, \top \}$.}

The following properties are prominent in lattice theory, and are available
for some of the highly regular lattices and posets that appear there (see
e.g.\ Aigner \cite{AiM79}).

%
%
\begin{defn}[(Rank Function and Graded Posets)]
For a top-bounded poset $\poset$ with $\top \in P$, a function
$\func{\rho}{P}{\N_{\height-1}}$ is a {\bf rank function} when $\rho(\top) =
0$ and $\forall a \cover b \in P, \rho(a) = \rho(b) - 1$. A poset $\poset$ is
{\bf graded}, or {\bf fully graded}, if it has a rank function.
\end{defn}
%

Let $\chains(\poset) \sub \pow{P}$ be the set of all maximal chains of
$\poset$. We assume that $\poset$ is bounded, therefore $\forall C \in
\chains(\poset), \bot,\top \in C$. 
	We will refer to the {\bf spindle chains} of a poset $\poset$ as
the set of its maximum length chains
	\[ \spindle(\poset) \define
		\left\{ C \in \class( \poset ) \st |C| =
			\height \right\}.	\]
	The {\bf spindle set}
	\[ I(\poset) \define \Un_{C \in \spindle(\poset)} C \sub P	\]
	is then the set of {\bf spindle elements}, including any elements which sit on
a spindle chain. Note that if $P$ is nonempty then there is always at least
one spindle chain and thus at least one spindle element, so $\spindle(\poset),
I(\poset) \neq \emptyset$. In our example in \fig{allex2}, we have $\height =
5$, $|\chains(\poset)| = 6$, $\spindle(\poset) = \{ \bot \cover A \cover H
\cover K \cover \top \}$, and $S(J) = \height( \up J ) + \height(\down J ) -
1 = 4$.

Given two posets $\poset_1 = \tup{ P, \le_P }$ and $\poset_2 = \tup{ Q, \le_Q
}$, a function $\func{f}{P}{Q}$ is an \textbf{order embedding} if $\forall
a,b \in P$ we have $a \le_P b \iff f(a) \le_Q f(b)$. A weaker notion is that
of $f$ being an \textbf{order homomorphism}
if $\forall a \le_P b \in P, f(a) \le_Q f(b)$. We can also say that $f$
\textbf{preserves the order} $\le_P$ into $\le_Q$, and is an \textbf{isotone}
mapping from $\poset_1$ to $\poset_2$. If $\forall a <_P b \in P, f(a) <_Q
f(b) \in Q$ then we say that $f$ does all this \textbf{strictly}. If instead
$f$ is an order homomorphism from $\poset_1$ to the dual $\tup{ Q, \ge_Q }$,
then we say that $f$ \textbf{reverses the order}, or $f$ is an
\textbf{antitone} mapping. If $\func{f}{P}{Q}$ is an order homomorphism, then
we can denote $f(\poset_1) \define \tup{ f(P), \le_f }$ as the
\textbf{homomorphic image} of $\poset_1$, with
	\[ f(P) \define \{ f(a) \st a \in P \} \sub Q,	\quad
		\le_f\, \, \define
            \leq_Q\!\left|_{f(P)\times f(P)}\right. \]
When clear from context, we will simply re-use $\le$ as the relevant order to
its base set, e.g.\ for an isotone $\func{f}{P}{Q}, a \le b \in P \implies
f(a) \le f(b) \in Q$. A \textbf{linear extension} of a poset $\poset$ is a
chain, $\mathcal{C}$, for which there is an order homomorphism
$\func{f}{P}{C}$.

The {\bf dimension} \cite{BaKFiP72,TrW92} of a poset $\poset = \tup{ P, \le
}$, denoted $\dim(\poset)$, is the minimum number, $m$, of total orders,
$\le^T_i$, such that there is an order embedding of $\le$ into $\Int_{i=1}^m
\le^T_i$.  Dimension is related to width, in that $\dim(\poset) \le {\cal
W}(\poset)$, but is a fundamentally different concept. The following theorem
from \cite{TrW92}, is taken to be the definition of dimension by some
\cite{OrO62}.

\begin{thm}\cite{TrW92}\label{dimN}
Let $\po=\tup{P,\leq}$ be a poset. Then $\dim(\po)$ is the least $t$ for
which $\po$ is isomorphic to an induced subposet of $\R^t = \tup{\R \times
\cdots \times \R, \leq \times \cdots\times\leq}$.
\end{thm}

Finally, we recognize $\tup{ \N, \le }$ as a total order using the normal
numeric order $\le$, and observe that for any bounded poset $\poset$, the
functions $\height(\up \cdot), \height(\down \cdot): P \rightarrow
\N_{\height}$ induce strict antitone and isotone order morphisms,
respectively. That is, if $a \cover b$ then
\[ \height(\up a) \geq \height(\up b)+1 > \height(\up b) \qquad\text{and}\qquad
    \height(\down a) < \height(\down a)+1 \leq \height(\down b). \]

\subsection{Interval Orders}\label{intervals}

Since our rank functions are interval-valued, we explicate the concepts
surrounding the possible ordering relations among intervals (see
\cite{JoCHoE14} for a full discussion). $\tup{\R,\le}$
is a total order, so for any $x_* \le x^* \in \R$, we can denote the (real)
interval $\bar{x} = [ x_*, x^* ]$, and $\overline{\R}$ as the set of all real
intervals on $\R$, so that $\bar{x} \in \overline{\R}$. Additionally, for $N
\in \N$ let $\overline{N}$ be the set of all intervals whose endpoints are
nonnegative integers $\le N$.


We now define the two most common interval orders (strong, subset), and a
third less common order (weak) which we focus on in this paper.

	\begin{defn}[(Strong Interval Order)] \label{strong}
	Let $<_S$ be a strict order on $\overline{\R}$ where $\bar{x} <_S
\bar{y}$ iff $x^* < y_*$. Let $\le_S$ be the reflexive closure of $<_S$, so
that $\bar{x} \leq_S \bar{y}$ iff $x^* < y_*$ or $\bar{x} = \bar{y}$.
	\end{defn}

	\begin{defn}[(Subset Interval Order \cite{TaP96b})]
	Let $\sub$ be a partial order on $\overline{\R}$ where $\bar{x}
\sub \bar{y}$ iff $x_* \ge y_*$ and $x^* \le y^*$.
	\end{defn}

	\begin{defn}[(Weak Interval Order)] \label{weak}
	Let $\le_W$ be a partial order on $\overline{\R}$ where $\bar{x} \le_W
\bar{y}$ iff $x_* \le y_*$ and $x^* \le y^*$. This is equivalent to the
product order $\leq \times \leq$ on $\R^2$.
	\end{defn}

Notice that the weak order $\leq_W$ and the subset order $\subseteq$ are
nearly conjugate orders \cite{BaKFiP72,DuBMiE1941}. First, Papadakis and
Kaburlasos \cite{PaSKaV10} observe that where our weak order $\le_W$ on
intervals is the product order on the reals $\le \times \le$, the subset
order $\sub$ is also a product order, but conjugately as $\ge \times \le$.
But $\leq_W$ and $\subseteq$ are not truly conjugate, since it is possible
for two distinct intervals to be comparable in both $\leq$ and $\subseteq$ if
there is equality at one of the endpoints. As an example, $[3,4]\subseteq
[2,4]$ and $[2,4] \leq_W [3,4]$.

One can ask whether the strong order $\le_S$ also has a conjugate, and the
answer appears to be yes, at least in small cases. We have experimentally
identified numerous conjugates and pseudo-conjugates (in the sense in which
$\le_W$ and $\subseteq$ are pseudo-conjugates) for $\le_S$ on the set of
integer intervals with endpoints between 1 and 4, but none appear to have a
compactly expressible structure, as both $\le_W$ and $\subseteq$ do
(see further discussion below in \sec{conclude}).

\subsection{Interval Rank}

For a full development of the concept of interval rank see
\cite{JoCHoE14}, here we present the definitions needed for the rest of
this paper. We begin by defining general interval rank functions, strict
interval rank functions, and a standard interval rank function, a
particularly significant strict interval rank function.

\begin{defn}[(Interval Rank Function)]\label{IntRank}
Let $\poset = \tup{P,\le}$ be a poset and $\sqsubseteq$ an order on real
intervals.  Then a function, $\func{R_\sqsubseteq}{P}{\overline{\N}}$, with
$R_\sqsubseteq(a) = [ r_*(a), r^*(a) ]$ for $a \in P$, is an {\bf interval
rank function for $\sqsubseteq$} if $R_\sqsubseteq$ is a strict order
homomorphism $\po \mapsto \tup{\overline{\N},\sqsubseteq}$, i.e., for all $a
< b$ in $\po$ we have $R_\sqsubseteq(a) \sqsubset R_\sqsubseteq(b)$. Let
${\bf R}_\sqsubseteq(\poset)$ be the set of all interval rank functions
$R_\sqsubseteq$ for the interval order $\sqsubseteq$ on a poset $\poset$.
\end{defn}

In \cite{JoCHoE14} we showed that the endpoints of these interval rank
functions are monotonic iff the interval order is one of $\{\le_W, \ge_W,
\subseteq, \supseteq\}$. In order to ensure that both endpoints of the
intervals are \emph{strictly} monotonic we must introduce the strict interval
rank function.

\begin{defn}[(Strict Interval Rank Function)]
An interval rank function $R_\sqsubseteq \in {\bf R}_\sqsubseteq(\poset)$ is
{\bf strict} if $r_*$ and $r^*$ are strictly monotonic. Let ${\bf
S}_\sqsubseteq(\poset)$ be the set of all strict interval rank functions
$R_\sqsubseteq$ for the interval order $\sqsubseteq$ on a poset $\poset$.
\end{defn}

Strict interval rank functions have the distinction of being the only
interval-valued functions which are strictly monotonic on both interval
endpoints. Each of the four types of monotonicity corresponds to an interval
order.
\begin{prop}\label{EndpointMonotoneStrict} \cite{JoCHoE14}
	Let $F : P \rightarrow \overline{\N}$ be an interval-valued
function on poset $\tup{\poset,\leq}$, so that $F(a) = [F_*(a),F^*(a)] \in
\overline{\N}$. Then, $F_*$ and $F^*$ are strictly monotonic functions iff
$F$ is a strict interval rank function for $\leq_W,\geq_W,\subseteq,$ or
$\supseteq$. In particular:
	\begin{enumerate}[(i)]
  \item $F_*$ and $F^*$ are both strictly isotone iff $\sqsubseteq= \le_W$.
  \item $F_*$ and $F^*$ are both strictly antitone iff $\sqsubseteq=
      \ge_W$.
  \item $F_*$ is strictly antitone and $F^*$ is strictly isotone iff
      $\sqsubseteq = \subseteq$.
  \item $F_*$ is strictly isotone and $F^*$ is strictly antitone iff
      $\sqsubseteq = \supseteq$.
\end{enumerate}
\end{prop}

We point out one particular strict interval rank function that is
especially useful, and call it the standard interval rank function.

\begin{defn}[(Standard Interval Rank)]
Let
\[ R^+(a) \define  [ \height(\up a)-1, \height - \height(\down a) ]
\in
\overline{\height-1} \]
be called the {\bf standard interval rank function}. For convenience we also
denote $R^+(a) = \left[ r^t(a), r^b(a) \right]$, where $r^t(a) \define
\height(\up a)-1$ is called the {\bf top rank} and $r^b(a) \define \height -
\height(\down a)$ is called the {\bf bottom rank}.
\end{defn}

The following proposition shows why we prefer $\Rpl$ over any other strict
interval rank function.
\begin{prop} \cite{JoCHoE14} \label{stdIntRankMaxSub}
For a finite bounded poset $\poset$, $R^+$ is maximal w.r.t. $\subseteq$ in
the sense that $\forall R \in {\bf S}_{\ge_W}(\poset)$ with $R(\poset)
\subseteq \overline{\height-1}$, $\forall a \in P, R(a) \sub R^+(a)$.
\end{prop}

\section{Conjugate Interval Rank}\label{conjIntRank}

Just as there is a privileged strict interval rank function for the
weak interval order $\geq_W$,
there is also one for the subset interval order $\subseteq$. We will call this the {\bf conjugate
standard interval rank function} because it is an interval rank function for
the interval order $\subseteq$, a conjugate interval order to $\geq_W$ for
which the standard interval rank is defined, and because its properties
closely mirror those of the standard interval rank function.

\begin{defn}[(Conjugate Standard Interval Rank)]
Let
	\[ \RplConj(a) = [ \rConjT(a), \rConjB(a) ] \define
		[\height(\upa)-1,\height + \height(\dna)-2] \] be called the {\bf
conjugate standard interval rank function}.	
\end{defn}

In \fig{bothRanks} we show the standard interval rank values along with the
conjugate standard interval rank values for the same example poset. We can
prove a result for $\RplConj$ analogous to that found in Proposition
\ref{stdIntRankMaxSub} for $R^+$, plus two other properties.

	\begin{figure*}[htbp]
	\begin{center}
	\subfloat[Example poset with $\Rpl(p)$ for each $p \in \po$.]{
	\begin{tikzpicture}[>=latex, line width=0.75pt]
	\bddExWithStdRanks
	\end{tikzpicture}}
	\hspace{1cm}
	\subfloat[Example poset with $\RplConj(p)$ for each $p \in \po$. ]{
	\begin{tikzpicture}[>=latex, line width=0.75pt]
	\bddExWithConjStdRanks
	\end{tikzpicture}
	}
	\caption{}
	\label{bothRanks}
	\end{center}
	\end{figure*}

\begin{prop}\label{conjugate}
For a finite bounded poset $\poset$,
\begin{enumerate}[(i)]
\item $\RplConj \in {\bf S}_{\subseteq}(\poset)$ is a strict interval rank
    function for the subset interval order $\subseteq$;
\item $\RplConj(\top) = [0,2(\height-1)];$ $\RplConj(\bot) = [ \height-1,
    \height-1]$;
\item $\RplConj$ is minimal w.r.t. $\leq_W$ in the sense that $\forall R
    \in {\bf S}_{\subseteq}(\poset)$ with $R(\poset) \subseteq
    \overline{2(\height-1)}$, $\forall a \in P, R(a) \ge_W \RplConj(a)$.
\end{enumerate}
\end{prop}

\begin{proof}\mbox{}
\begin{enumerate}[(i)]
  \item We first show that $\rConjT(a) \leq \rConjB(a)$ so that $\RplConj$
      is an interval-valued function. We claim that $\height(\up a) -1 \leq
      \height - \height(\down a)$. Indeed, $\forall a \in P$ we have that
      \[ \height(\down a) + \height(\up a) \leq \height([0,1])+1 = \height+1 \]
      with equality iff $a \in I(\poset)$ is a spindle element. Rearranging we
      have the desired $\height(\up a) -1 \leq \height - \height(\down a)$.
      Then, if we increase only the right side by a positive value we
      retain the inequality. Since $\height(\down a) \geq 1$ we can add
      $2\height(\down a)-2 >0$ to the right side. This yields
      \[ \height(\up a) -1 \leq \height - \height(\down a) +2\height(\down a)-2 = \height+\height(\down a)-2. \]
      Therefore, $\RplConj$ is an interval-valued function.

      Then, it is evident from $\height(\upa)$ being strictly antitone, and
    $\height(\dna)$ being strictly isotone that $\rConjT$ is strictly
    antitone and $\rConjB$ is strictly isotone. Therefore, by Proposition
    \ref{EndpointMonotoneStrict} we have that $\RplConj$ is a strict
    interval rank function for $\subseteq$.
  \item Follows from $\height(\up \top) = \height(\down \bot) = 1$ and
      $\height(\down \top) = \height(\up \bot)=\height$.

  \item Since any $R \in {\bf S_{\sub}}(\poset)$ is a strict interval rank
      function for $\subseteq$ we know that
    if $a<b \in P$, $R(a) \subset R(b)$. Therefore
    \[ r_*(a) > r_*(b) \text{ and } r^*(a) < r^*(b) \]
    so that $r_*$ is strictly antitone and $r^*$ is strictly isotone. In
    addition, we are restricting to the case where $0 \leq r_*\leq r^* \leq
    2(\height-1)$. Under these assumptions we must show that $\forall a \in
    P$
    \begin{align*}
    r_*(a) &\geq \height(\up a)-1\\
    r^*(a) &\geq \height+\height(\down a)-2.
    \end{align*}

    First notice that, by definition of $\height$, there must be a chain
    $C_{\height} \subseteq P$ of length $\height$ with greatest element
    $\top$ and least element $\bot$. Since $R(\poset) \subseteq
    \overline{2(\height-1)}$, $r_*$ is strictly antitone, $r^*$ is strictly
    isotone, and $r_* \leq r^*$ we must have $R(\bot) =
    [\height-1,\height-1]$ (so that we can decrease $r_*$ by one along
    $C_{\height}$ and stay positive, and increase $r^*$ by one and stay
    less than $2(\height-1)$ as we go from $\bot$ to $\top$) and $R(\top) =
    [0,2(\height-1)]$.

    Now, let $a \in P$, by definition of $\height(\cdot)$ we know that
    there must be a chain $C \subseteq P$ of length $\height(\up a)$ with
    greatest element $\top$ and least element $a$. We already showed that
    $r_*(\top) = 0$. Then, in order for $r_*$ to be strictly antitone we
    need $\forall c_1 < c_2 \in C$, $r_*(c_1) > r_*(c_2)$. Therefore
    $r_*(c)$ must be at least the chain distance from $\top$ to $c$ along
    $C$, less one, for all $c \in C$. In particular, $r_*(a) \geq
    \height(\up a) -1$.

    Dually, there must be a chain $D \subseteq P$ of length $\height(\down
    a)$ with greatest element $a$ and least element $\bot$. We already know
    $r^*(\bot)=\height-1$. In order for $r^*$ to be strictly isotone it
    must be true that $\forall d_1 < d_2 \in D$, $r^*(d_1) < r^*(d_2)$.
    Therefore, $r^*(d)$ must be at least $\height-1$ plus the chain
    distance from $\bot$ to $d$ along $D$, less one, for all $d \in D$. In
    particular
    \[ r^*(a) \geq \height-1 + (\height(\down a)-1) = \height+\height(\down a)-2.\]
\end{enumerate}
\end{proof}

One might ask why $R^+$ is called ``standard''
when its conjugate $\RplConj$ is available. First, the behavior of $R^+$ is
much more natural and meets our criteria of advancing monotonically with
level, where the intervals for $\RplConj$ ``nest''. We will see in
Proposition \ref{conjIso} below that the homomorphic image under the
conjugate interval rank $\RplConj$ is isomorphic to that of the standard
interval rank $R^+$, and thus does not bring any particular value compared to
$R^+$. And finally, $\RplConj$ does not conform to our desired criteria of
ranks being in the set $\{0,1,\ldots,\height-1\}$, but rather being in
$\{0,1,\ldots,2(\height-1)\}$.

\section{The Homomorphic Image of Standard Interval Ranks}

A general interval rank function $R$ takes elements $p \in P$ of a poset to
intervals $R(p) \in \overline{\height-1}$. But from the discussion in
\sec{intervals}, we know that these intervals $R(p)$ can also be ordered. In
this section we consider the behavior and properties of standard interval
rank $R^+$ as an order morphism.

The rank intervals $R^+(a)$ are themselves elements in a poset, that is in the
homomorphic image $R^+(\poset) = \tup{ R^+(P), \le_W }$. Thus they also have
an interval rank structure within $R^+(\poset)$. We can analyze the structure
of this homomorphic image, including changes in height, width, and dimension
compared to the underlying poset $\tup{P,\le}$. And finally, $R^+$ as an
order morphism can be iteratively applied to derive $R^+(R^+(\poset))$, etc.,
and we can examine the long-term properties of this iterated
application of $R^+$.

\subsection{The Interval Rank Poset} \label{structureRpl}

Recalling the definition of the homomorphic image of a poset given in Section
\ref{OrderedSetsDefns}, we will refer to $\RplP$ as the \emph{interval rank
poset of $\po$}. From Definition \ref{IntRank} we know that $\Rpl$ is a
strict order homomorphism into $\tup{\overline{N},\geq_W}$. Therefore,
$\RplP$ is an induced subposet of $\overline{\height-1}$ with $\geq_W$ as its
ordering. \fig{homo} shows the homomorphic image of the poset found in
\fig{allex2}. It is easily verified that there is a strict order homomorphism
from $\poset$ to $R^+(\poset)$. But notice that two of the elements of $\po$,
$J$ and $E$, had the same standard interval rank, so they collapse into a
single element, $[2,3]$, in $\RplP$. These behaviors will be considered in
detail below in \sec{iterating}.

The structure of $R^+(\poset)$ allows us to identify elements $a,b \in P$ which
are either comparable $R^+(a) \sim_W R^+(b)$ or noncomparable $R^+(a) \|_W
R^+(b)$ in terms of the weak interval order relation $\le_W$ between their
interval ranks. If they are noncomparable in the weak order they are thereby
comparable in the conjugate subset order, so that $R^+(a) \sub R^+(b)$ or
$R^+(a) \supseteq R^+(b)$.

\begin{figure}[htbp]
  \begin{center}
    \includegraphics[scale=0.225]{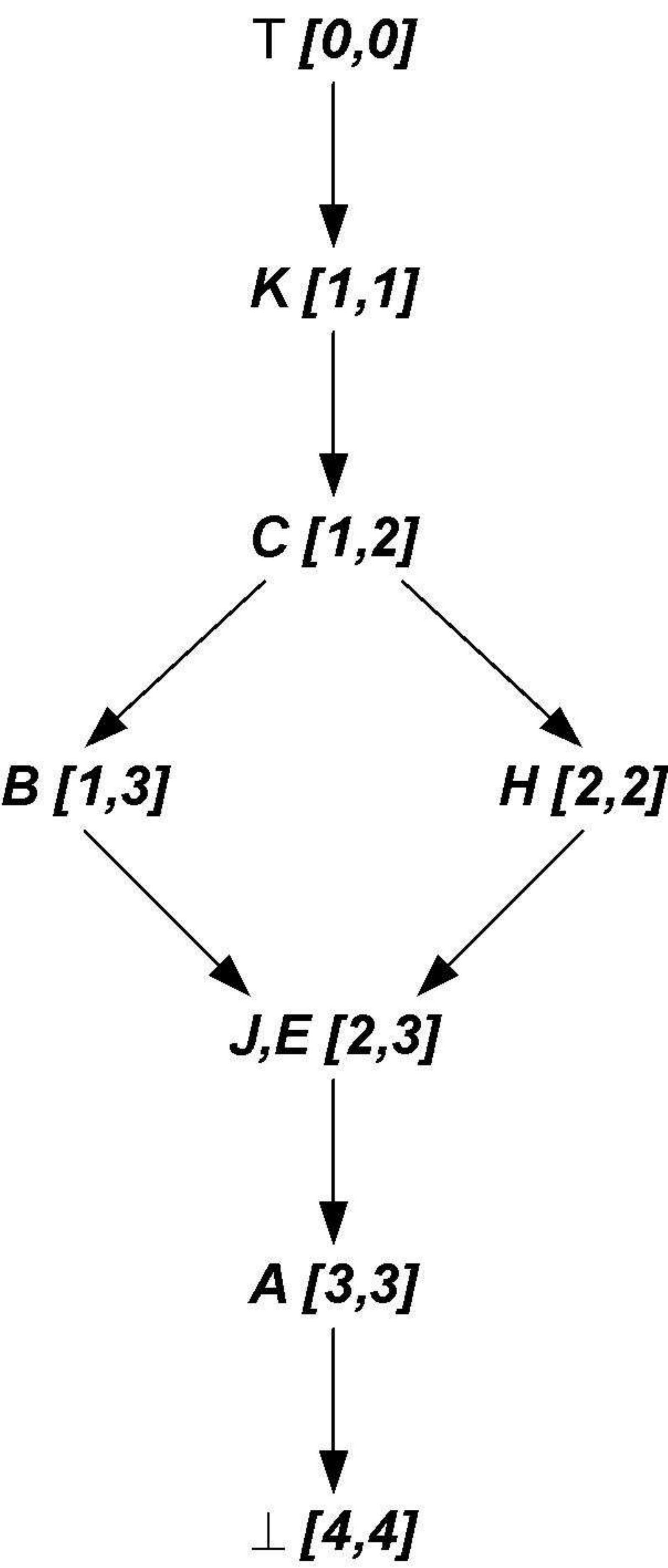}
  \end{center}
  \caption{The homomorphic image $R^+(\poset)$ induced from the
example in \fig{allex2} showing the weak interval order $\le_W$ on the
standard interval ranks $R^+(a) \in \overline{N}$.}\label{homo}
\end{figure}



\fig{homoboth1} now shows the example in \fig{bothRanks} equipped with both
standard interval rank and the edges in the interval rank poset
$R^+(\poset)$, shown in dashed lines, with $J$ and $E$ identified as a new
contracted element in $R^+(\poset)$ with a dashed oval. 
Note how the dashed edges of the interval rank poset proceed vertically very
tightly from top to bottom, linking elements with the closest standard interval
ranks, whether those element pairs are in $\poset$ or not. 

\begin{figure}[htbp]
  \begin{center}
    \includegraphics[scale=0.275]{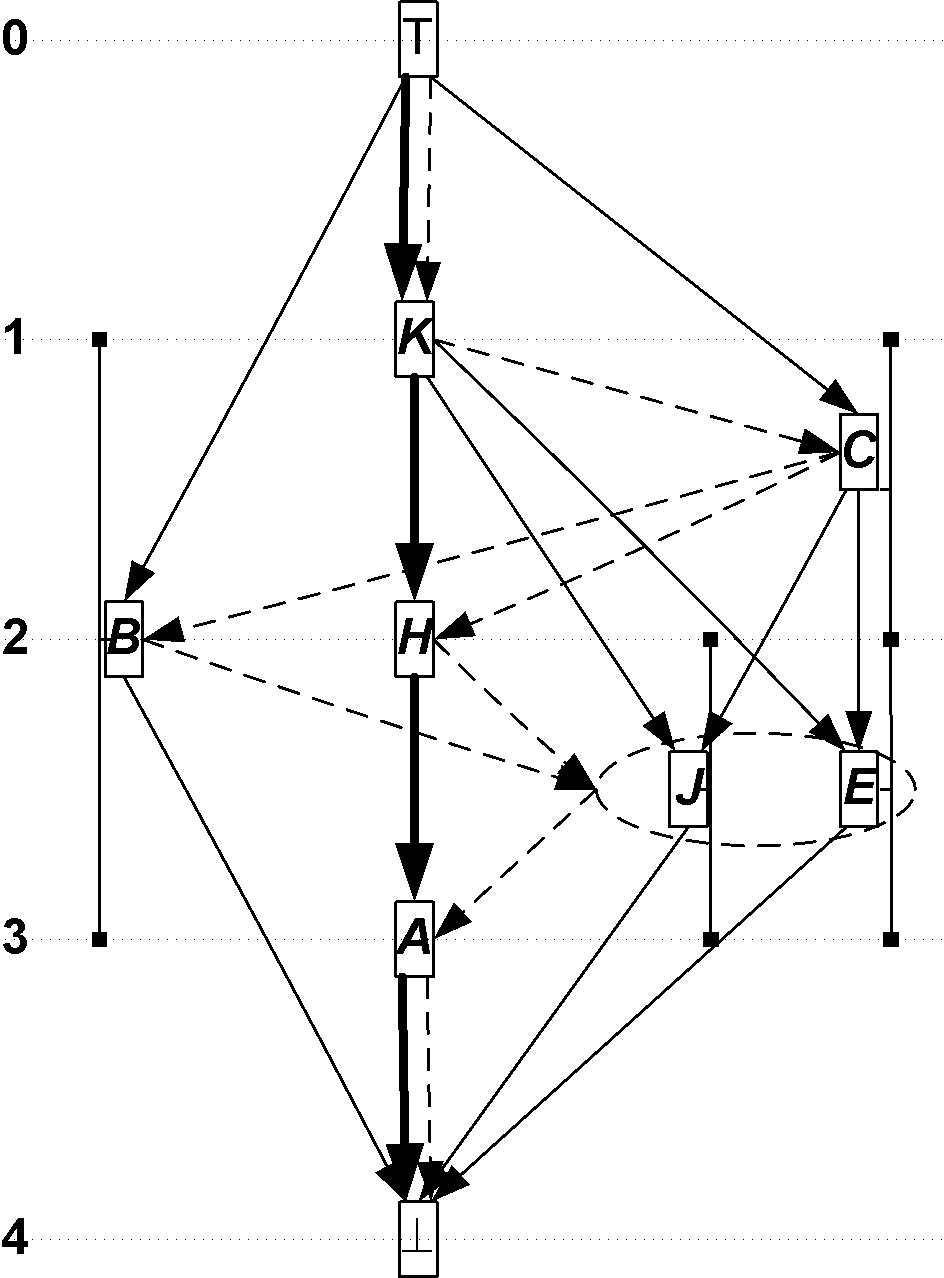}
  \end{center}
  \caption{Example poset from \fig{allex2} equipped with
interval ranks $R^+(a)$, homomorphic image links in dashed lines, and
separations and widths on all links.}\label{homoboth1}
\end{figure}


As mentioned in Section \ref{conjIntRank}, we now prove that the homomorphic
image of a poset, $\po$, under the conjugate interval rank $\RplConj$ is
isomorphic to that of the standard interval rank $R^+$. Therefore, it is
enough to study the structure of just the homomorphic image of $R^+$.

\begin{prop}\label{conjIso}
$\RplP \cong \RplConj(\po)$.
\end{prop}
\begin{proof}
To show that these two posets are isomorphic we will show that there is an
order embedding, $\varphi : \RplP \rightarrow \RplConj(\po)$, that is
surjective. Define
\begin{alignat*}{4}
\varphi: &\RplP &\longrightarrow &\RplConj(\po)\\
		 &[x,y] &\longmapsto     &[x,2(\height-1)-y].
\end{alignat*}

We will first show that $\varphi$ is surjective, i.e., that for every $[a,b]
\in \RplConj(\po)$ there is an $[x,y] \in \RplP$ so that $\varphi([x,y]) =
[a,b]$. Consider $[a,b] \in \RplConj(\po)$. Then there is some $p \in \po$
such that $\height(\upp)-1 = a$ and $\height+\height(\dnp)-2 = b$. Let $[x,y]
= [\height(\upp)-1, \height-\height(\dnp)]$. Clearly this is an element of
$\RplP$ by the definition of the $R^+$ operator; in fact, it is $R^+(p)$.
Then the image of $[x,y]$ under $\varphi$ is
\begin{align*}
\varphi([x,y]) &= \varphi([\height(\upp)-1,\height-\height(\dnp)])\\
    &= [\height(\upp)-1,2(\height-1)-(\height-\height(\dnp))] = [a,b]
\end{align*}
	So, for every $[a,b] \in \RplConj(\po)$ we have an $[x,y] \in \RplP$ such
that $\varphi([x,y]) = [a,b]$. Therefore, $\varphi$ is surjective.

To show that $\varphi$ is an order embedding we must show that for all
$[x,y], [z,w] \in \RplP$ we have $[x,y] \geq_W [z,w] \iff \varphi([x,y])
\subseteq \varphi([z,w])$. First we show the forward direction. Assume we
have $[x,y] \geq_W [z,w]$, then $x \geq z$ and $y \geq w$. We map these
intervals to $\varphi([x,y]) = [x,2(\height-1)-y] = [a,b]$ and
$\varphi([z,w]) = [z,2(\height-1)-w]=[c,d]$. Obviously $a \geq c$ since $a=x$
and $c=z$. Then, since $y \geq w$ we have $-y \leq -w$ which means $b \leq
d$. Putting this together we have $\varphi([x,y])=[a,b] \subseteq
[c,d]=\varphi([z,w])$ which proves the forward implication.

Now, we assume that $[a,b]=\varphi([x,y]) \subseteq \varphi([z,w])=[c,d]$, so
$a \geq c$ and $b \leq d$. Because $a=x$ and $z=c$ we have $x \geq z$. Then,
$b = 2(\height-1)-y \leq 2(\height-1)-w = d$ so $-y \leq -w$ and then $y \geq
w$. Putting this together we get $[x,y] \geq_W [z,w]$ as desired.

\end{proof}

\subsection{Properties of the Interval Rank Poset} \label{properties}

We now consider the interval rank structure of $R^+(\poset)$, the homomorphic
image of $\po$, itself, and further iterations, $R^+(R^+( \ldots R^+(\poset)
\ldots ))$, thereof.

First, when we compare $\po$ to its homomorphic image $\RplP$ we observe that
the height always increases while the width decreases.

\begin{prop}\label{heightIncr}
The height of the interval rank poset is greater than the height of the
ordered set itself, i.e., $ \height(\RplP) \geq \height(\po)$.
\end{prop}
\begin{proof}
Let $S$ be a spindle chain in $\po$. The elements in $S$ have rank intervals
$[i,i]$ where $i$ takes all integer values between 0 and $\height(\po)-1$
(inclusively). The image of this spindle chain $S$ under $R^+$ is a (not
necessarily saturated) chain, $S'$, in $\RplP$. From the definition of the
height of a poset we have that
\[\forall C \in {\cal C}(\RplP), \quad  \height(\RplP) \geq |C|.\]
Thus, $\height(\RplP)\geq|S'| = |S| = \height(\po)$.
\end{proof}

\begin{prop}\label{widthDecr}
The width of the interval rank poset is less than the width of the ordered
set itself, i.e., $\width(\RplP) \leq \width(\po)$.
\end{prop}
\begin{proof}
Let $\mathcal{A}$ be the set of all antichains in $\poset$, not necessarily
maximal; and $\mathcal{A}^+$ the set of all antichains in $\RplP$, also not
necessarily maximal. Let $A \in \mathcal{A}^+$. Consider the set of preimages
of elements in $A$ w.r.t. the map $R^+$,
\[\preA = \{p \in \po : R^+(p)\in A\}.\]
Note that it is of course possible that $|\preA| \geq |A|$. We claim that
$\preA \in \mathcal{A}$. If not then there are two elements, $p,q\in \preA$
such that $p< q$. Since $R^+$ is an order homomorphism we have that $R^+(p) <
R^+(q)$. But this is a contradiction to $A$ being an antichain, so $\preA$
must be an antichain in $\po$. Therefore we have the following chain of
inequalities on width:
\begin{align*}
\width(\RplP) = \max_{A \in \mathcal{A}^+} |A| &\leq \max_{A \in \mathcal{A}^+} |\preA|\\
        &\leq \max_{A' \in \mathcal{A}} |A'| = \width(\po).
\end{align*}
\end{proof}

Notice that in the proof of Proposition \ref{heightIncr} we were able to
start with a chain (namely the spindle chain) in $\po$ and use the fact that
its image in $\RplP$ is a chain. This is because $\Rpl$ is an order morphism
to the reversed interval order, so the images of comparable elements are
comparable. However, there is no equivalent property that we could use in the
proof of Proposition \ref{widthDecr}. Every antichain in $\mathcal{A}^+$ must
come from an antichain in $\mathcal{A}$, but some antichains in $\mathcal{A}$
map to non-antichains in $\RplP$.

Given Theorem \ref{dimN}, which states that an $n$ dimensional poset is one
which is an induced subposet of $\R^n$ with ordering relation $\leq^n$, but
not an induced subposet of $\R^{n-1}$, we can now easily see that the
dimension of $\RplP$ is at most 2.

\begin{cor}\label{dim2} For any poset $\po = \tup{P,\leq}$ its homomorphic image under the standard
interval rank function, $R^+(\po)$, has
\[\dim(R^+(\po)) \leq 2\]
\end{cor}

\begin{proof}
By definition, $\RplP$ is a subposet of $\R^2$ with the reversed product
order, so by Theorem \ref{dimN} we see that $\RplP^*$ (the dual of $\RplP$)
has dimension at most 2. It's clear that a poset and its dual have the same
dimension, so we have that $\dim(\RplP) \leq 2$.
\end{proof}

Note that it may be the case that $\RplP$ has dimension one (i.e., it is a
chain). For example, the following proposition gives a particular sufficient
condition on $\po$ for $\RplP$ to be a chain.

\begin{prop}\label{gradedChain}
If $\po$ is graded then $\RplP$ is a chain.
\end{prop}

In order to prove this we must cite a part of a proposition proved in
\cite{JoCHoE14}.
\begin{prop} \label{prop1}
Let $\po=\tup{P,\leq}$ be a poset such that $\po$ is bounded, and $|P| \geq
2$. For an element $a \in P$, the width of its standard interval rank is zero iff
$a$ is a spindle element. I.e., $W(R^+(a)) = 0 \iff a \in I(\poset)$.
\end{prop}

\begin{proof}[Proof of Proposition \ref{gradedChain}]
Every element in a graded poset sits on a spindle chain. This is simply because
in a graded poset every maximal chain is the same length. Now, from
Proposition \ref{prop1} we know that if $a \in I(\po)$ then $W(R^+(a))=0$.
Therefore, the only elements of $\RplP$ are the trivial intervals $[i,i]$ for
$0\leq i \leq \height-1$. This set forms a chain under $\geq_W$ (the dual of
the product order on $\R^2$).
\end{proof}

However, there are ungraded posets, $\po$, for which $\RplP$ is a chain. For
example, the poset $N_5$, consisting of a length 4 chain and a length 3 chain
which share their top and bottom elements
, is ungraded and $R^+(N_5)$ is a
chain. Also, whether the dimension of $\RplP$ is 2 or 1 does not depend on
the dimension of $\po$. There are posets of dimension greater than 2 whose
interval rank poset has dimension 1: any boolean $n$-cube has dimension $n$
and is graded, thus its interval rank poset is a chain. In addition, there
are posets of dimension 2 whose interval rank poset also has dimension 2.

\subsection{Iterating $\Rpl$}\label{itrRpl} \label{iterating}

The fact that height is non-decreasing and width non-increasing from $\po$ to
$\Rpl$ leads us to ask the following question: what happens when we
repeatedly apply the $\Rpl$ operator? If height \emph{strictly} increases and
width \emph{strictly} decreases then it's clear that we end up with a chain
if we apply $\Rpl$ enough times. However, Propositions \ref{heightIncr} and
\ref{widthDecr} cannot be reformulated with strict inequalities so this
chain conjecture is not obvious. Based on experimental evidence (see Section
\ref{ExpAspect}) it appeared that when we apply $\Rpl$ enough times the
result is a chain. This turned out to be true, which we now prove.

For Lemma \ref{AllEventuallyGrad} we will need to define the poset
$\RplAllP$. See \fig{allexrank_int2} for an example, and compare to Figure
\ref{homo} which contains $\RplP$ for the same $\po$.

\begin{defn}
Given a poset $\po = \tup{P,\leq}$ define $\Rpl_{All}(\po) =
\tup{P,\leq_{R_A}}$ where $p <_{R_A} q$ iff $\Rpl(p) >_W \Rpl(q)$. Notice
that $<_{R_A}$ is a strict order, so we must take its reflexive closure to
create $\leq_A$. This is just $\RplP$ without identifying elements that have the
same interval rank. So if $\Rpl(p) = \Rpl(q)$ for $p,q \in \po$ with $p \neq
q$ we retain both elements and make them incomparable.
\end{defn}

\mytikz{\RplAllEx}{>=latex, line width=0.75pt}{allexrank_int2}{An example of
$\RplAllP$ where $\po$ is the poset found in \fig{allex2}. Compare to $\RplP$
found in \fig{homo}.}

Given this definition we will now prove two lemmas that will lead us to the
proof that iterating $\Rpl$ enough times yields a chain.

\begin{lem}\label{AllEventuallyGrad}
Given a poset, $\po=\tup{P,\leq}$, there exists an $m$ such that
$(\Rpl_{All})^m(\po)$ is graded.
\end{lem}

\begin{proof}
Fix an ungraded poset, $\po$, and consider the set of graded posets that
extend $\po$,
\[G(\po) = \{ \qo = \tup{P,\leq_{\qo}}:~ \leq~ \subset~ \leq_{\qo}, \qo \text{ graded}\}.\]
That is, if $p \leq q$ then $p \leq_{\qo} q$. Clearly this set is nonempty
since there is at least one linear extension of $\po$ (recall the definition
of a linear extension from Section \ref{OrderedSetsDefns}), and linear
extensions, being chains, are graded. Also, the number of comparisons
(ordered pairs $(a,b)$ such that $a\leq b$) in a total order only depends on
the number of elements in the chain. All pairs of elements are comparable,
and we have the reflexive comparisons, so
\[ \max_{\qo \in G(\po)} |\leq_{\qo}| = \binom{|P|}{2}+|P| = \frac{|P|^2+|P|}{2}.\]
Also, any poset $\mathcal{L} = \tup{P,\leq_{\mathcal{L}}}$ such that
$|\leq_{\mathcal{L}}|=\frac{|P|^2+|P|}{2}$ is a total order.

Now we claim that $\po$ being ungraded implies $|\leq| < |\leq_{R_A}|$. So
when we apply the $\Rpl_{All}$ operation we always end up with strictly more
comparable pairs of elements. If this is true then either
\begin{enumerate}
  \item at some point, iterating $\Rpl_{All}$ yields a graded poset whose
      partial order has strictly less than $\frac{|P|^2+|P|}{2}$ elements,
      or
  \item after iterating $\Rpl_{All}$ enough times we will get a partial
      order with exactly $\frac{|P|^2+|P|}{2}$ comparisons.
\end{enumerate}
We will show that every time we iterate $\Rpl_{All}$ on an ungraded poset we
add at least one comparison. If we iterate $\Rpl_{All}$ and get to
$\frac{|P|^2+|P|}{2}$ without hitting a graded poset up to this point, then
we are in case 2. Otherwise, we got to a graded poset with strictly less than
$\frac{|P|^2+|P|}{2}$ elements, and are in case 1. Either way, there is an
$m$ such that $(\Rpl_{All})^m(\po)$ is graded.

Finally, we must prove that $\po$ being ungraded implies $|\leq| <
|\leq_{R_A}|$, i.e., that we gain at least one comparison. Choose $p \in P$
with $\Rpl(p)=[x,y]$ for some $x,y\in \N$ such that $x\neq y$. There must be
at least one $p$ (if not, then $\po$ is graded). Then, choose $q \in P$ with
$\Rpl(q)=[x,x]$. Again, there must be at least one since $0\leq x\leq
\height(\po)-1$ and all $[z,z]$ with $0\leq z\leq \height(\po)-1$ are
represented on a spindle chain. Clearly, $p \nsim q$ in $\po$ because
$\height(\uparrow p)=\height(\uparrow q)$. But in $\Rpl_{All}(\po)$ we have
$q >_{R_A} p$ since $[x,x]\leq_W[x,y]$. So going from $\leq$ to $\leq_{R_A}$
we added at least one comparison.
\end{proof}

\begin{lem}\label{AllGradImpliesGrad}
If $(\Rpl_{All})^m(\po)$ is graded then $(\Rpl)^m(\po)$ is graded.
\end{lem}

\begin{proof}
Assume that $(\Rpl_{All})^m(\po)$ is graded. Then there is a rank function,
$r: P \rightarrow \N$ such that if $p\prec q$ in $(\Rpl_{All})^m(\po)$ then
$r(p)= r(q)+1$. Now, there may be some elements, $p_i$ in $P$ such that
$(\Rpl)^m(p_i)$ are all equal. So in $(\Rpl_{All})^m(\po)$ they all have the
same immediate parents ($\{a_j\}$) and children ($\{c_k\}$). Thus, $r(p_i)$
are all equal. So when we collapse all $p_i$ into one element, $p$, in $\RplP$
we have a non-ambiguous rank, $r(p)=r(p_i)$ for all $i$, for $p$.
\end{proof}

\begin{prop}\label{EventuallyChain}
Given a poset, $\po=\tup{P,\leq}$, there exists an $n$ such that
$\left(\Rpl\right)^n(\po)$ is a chain.
\end{prop}

\begin{proof}
Fix a poset $\po=\tup{P,\leq}$. By Lemma \ref{AllEventuallyGrad} there is an
$m$ such that $(\Rpl_{All})^m(\po)$ is graded. Then by Lemma
\ref{AllGradImpliesGrad}, $(\Rpl)^m(\po)$ is graded. Finally, by Proposition
\ref{gradedChain}, since $(\Rpl)^m(\po)$ is graded we have that
$(\Rpl)^{m+1}(\po)$ is a chain.
\end{proof}

The chain that we arrive at can be thought of as a \emph{total preorder} -- a
reflexive and transitive relation in which every pair of elements is
comparable \cite{DaBPrH90} -- which extends $\po$. Preorders are not
antisymmetric, so when two elements $p$ and $q$ are identified in the final
chain we will say that $p \leq q$ and $q \leq p$, but $p \neq q$. This
resulting total preorder is easily computable, and clearly calls out our
concept of ``levels'' in an ungraded poset. Figure \ref{totalPreorderEx}
shows the final total preorder for the example from \fig{allex2}.


\begin{figure}
  \centering
  \includegraphics{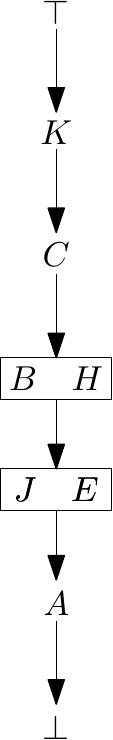}
  \caption{The canonical total preorder for the poset found in \fig{allex2}.
Notice that $B$ and $H$ are identified as well as $J$ and $E$. Looking at
\fig{allex2} it's easy to see why $J$ and $E$ are identified, but not
clear why $B$ and $H$ are.}\label{totalPreorderEx}
\end{figure}

Now that we have established that we always end up with a chain after
iterating $\RplP$ enough times, three interesting questions arise: how many
iterations does it take to end up with a chain, what is the final chain
length, and how much does the height increase. We were able to answer these
questions exactly for all bounded posets of size 3 to 9 (see Section
\ref{ExpAspect} for more on this). The averages\footnote{Note that there is
only one bounded poset of height 9 when we restrict to posets of size between
3 and 9, the chain of length 9. So, the average final height of chain for
bounded posets of size 9 is the average taken over only one sample.} are
collected in Table \ref{iterationsTable}.
\begin{table}[htbp]
\begin{center}
\begin{tabular}{|l||c|c|c|c|c|c|c|}
\hline
Bounded Poset Size				& 3 & 4    & 5     & 6     & 7     & 8     & 9 \\\hline\hline
Avg. \# of elements in chain    & 3 & 3.5  & 4.2   & 4.75  & 5.381 & 5.959 & 6.517  \\ \hline
Avg. \# of iterations to chain  & 0 & 0.5  & 0.8   & 1.00  & 1.127 & 1.236 & 1.335 \\ \hline
\multicolumn{8}{c}{}\\ \hline
Bounded Poset Height	     	& 3 & 4    & 5     & 6     & 7     & 8     & 9 \\\hline\hline
Avg. final height of chain	    & 3 & 4.348 & 6.068 & 7.092 & 7.806 & 8.409 & 9 \\\hline
\end{tabular}
\end{center}
\caption{Averages of iteration data for all bounded posets of size 3
to 9. Note that we are only looking here at posets of size between 3 and 9. In particular,
there is only one poset of size $\leq 9$ with height 9, the chain of length 9. This gives
us a misleading average final height of chain when the starting poset is height 9.}\label{iterationsTable}
\end{table}

For bounded posets of size greater than 9 we cannot generate all posets with
the computing resources available (there are only 2045 bounded posets of size
9 but 16999 of size 10, see A000112 in \cite{oeis} for values for larger
posets). Therefore, we generated many random bounded posets of sizes between
10 and 25 (200 posets of each size) to get similar data. See Section
\ref{ExpAspect} for a discussion on the method used to generate random
posets.

Figures \ref{SizeStats}, \ref{NumItersStats}, and \ref{HeightStats} show the
box plot and regression for the size of $\po$ vs. size of chain, size of
$\po$ vs. number of iterations to a chain, and height of $\po$ vs. height of
chain respectively. From the limited amount of data it appears that the the
number of iterations that must be done to get to the total preorder is
roughly linear in the number of elements in the poset. The compression that
we get in the end is approximately $\ln(n)	$.

\begin{figure}
\centering
\subfloat[Box Plot]{
\includegraphics[width=3.5in]{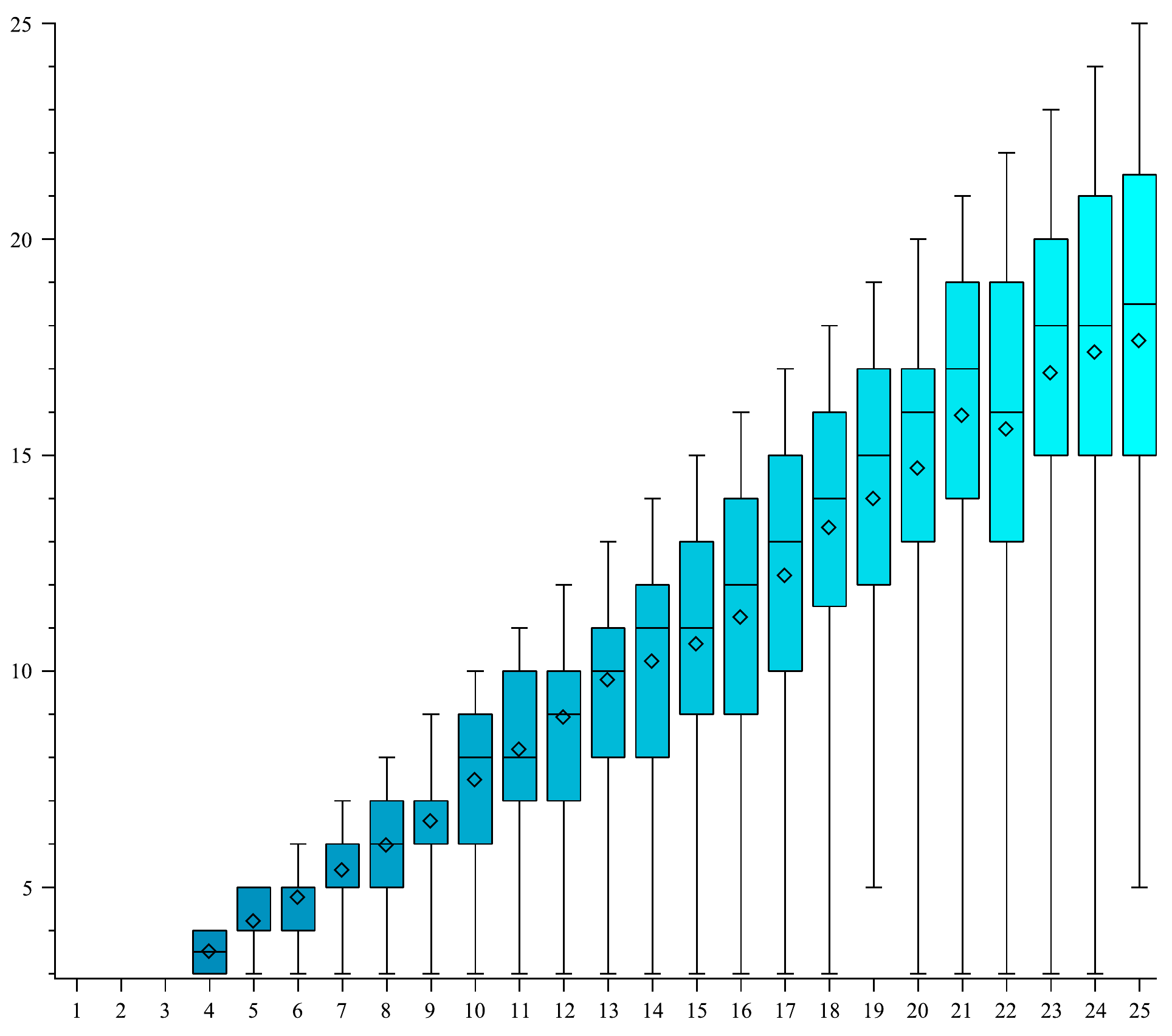}
}

\subfloat[Linear fit for average size of resulting chain as a function of
size of original poset: $y=0.7010 x+0.4854$, $R^2 = 0.9964$]{
\includegraphics[width=3.5in]{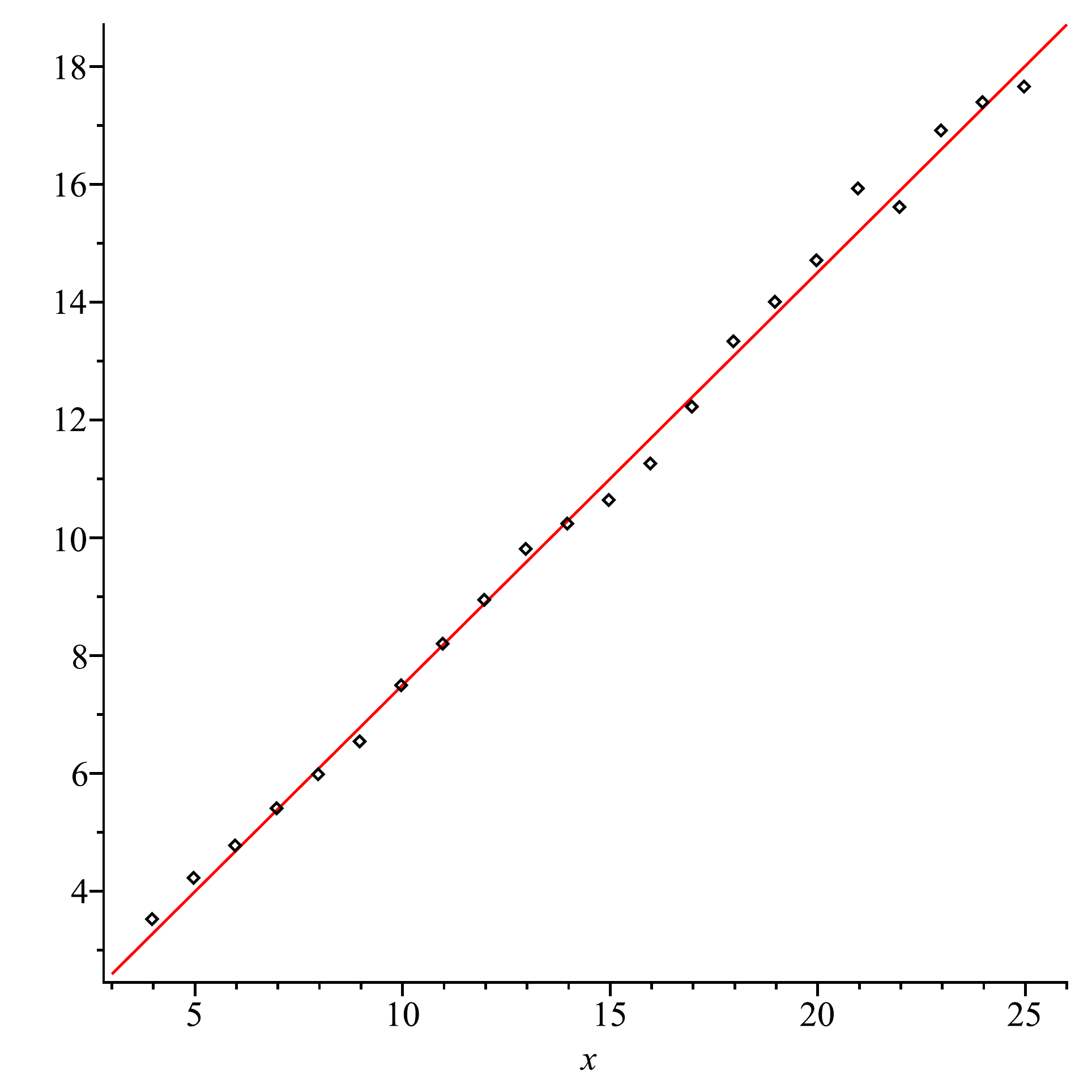}
} \caption{Statistics for the ``size of poset vs. size of resulting chain"
data}\label{SizeStats}
\end{figure}

\begin{figure}
\centering
\subfloat[Box Plot]{
\includegraphics[width=3.5in]{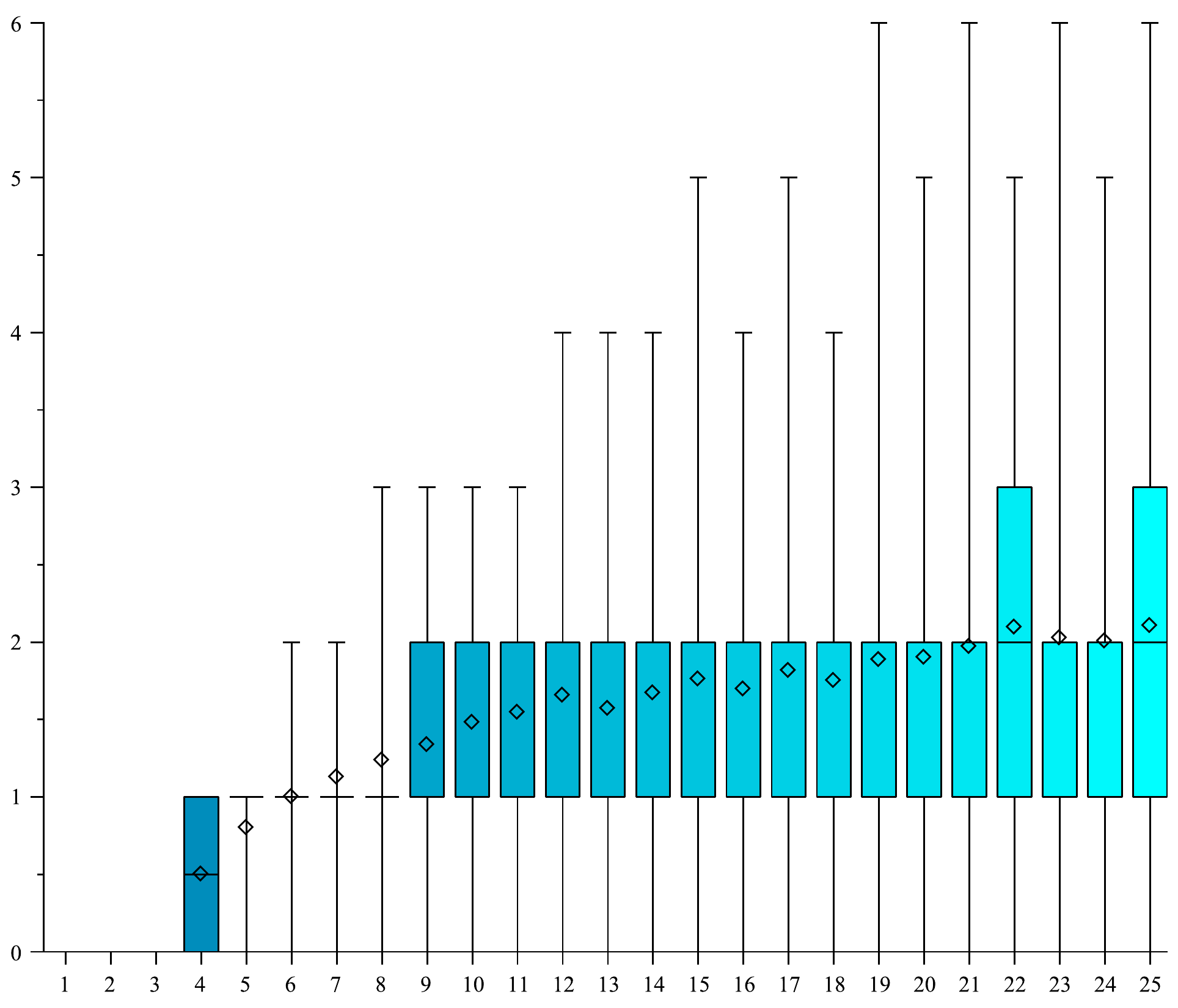}
}

\subfloat[Logarithmic fit for average number of iterations as a function of
size of original poset: $y=0.8003 \ln(x)-0.4574$, $R^2=0.9762$]{
\includegraphics[width=3.5in]{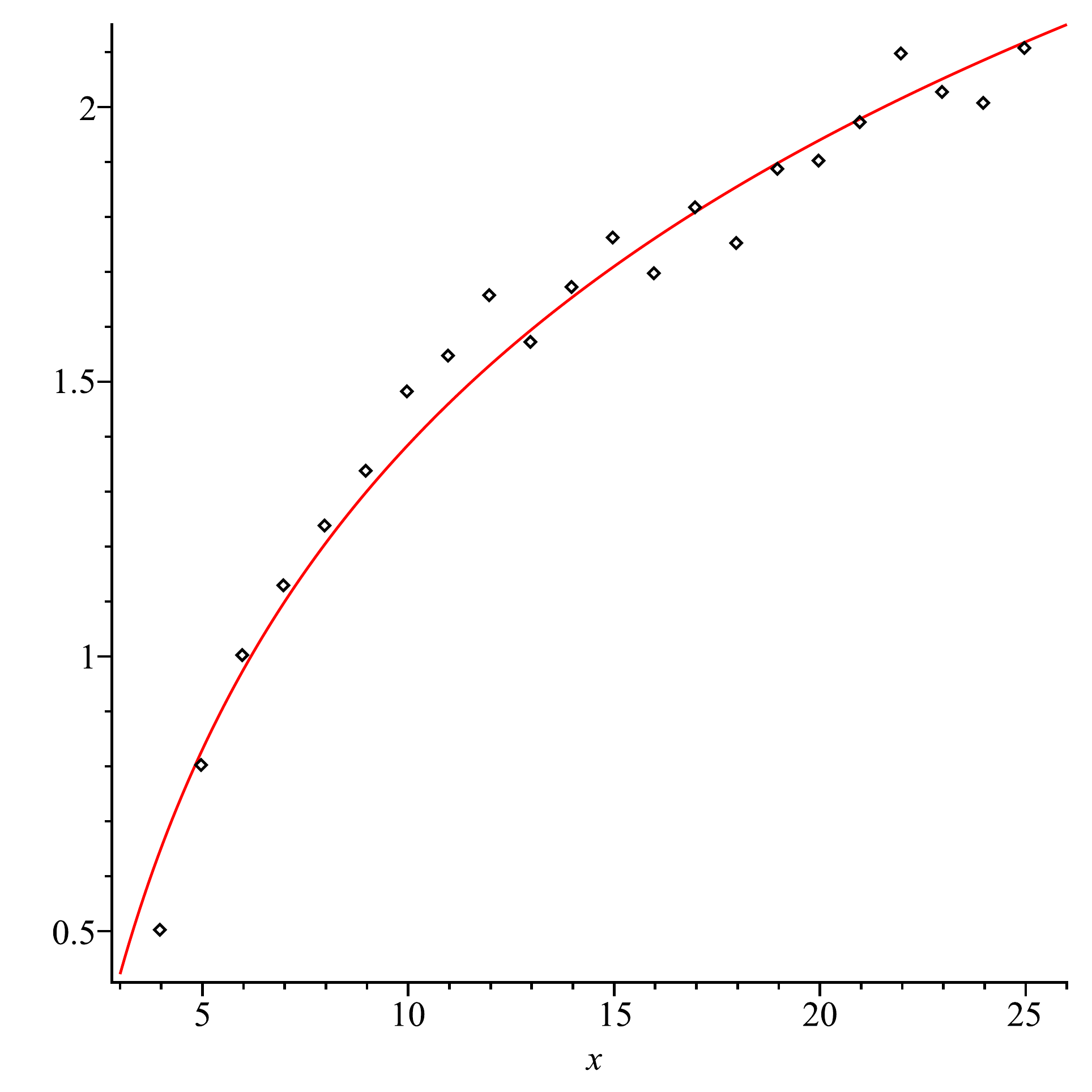}
} \caption{Statistics for the ``size of poset vs. number of iterations to
arrive at a chain" data}\label{NumItersStats}
\end{figure}

\begin{figure}
\centering
\subfloat[Box Plot]{
\includegraphics[width=3.5in]{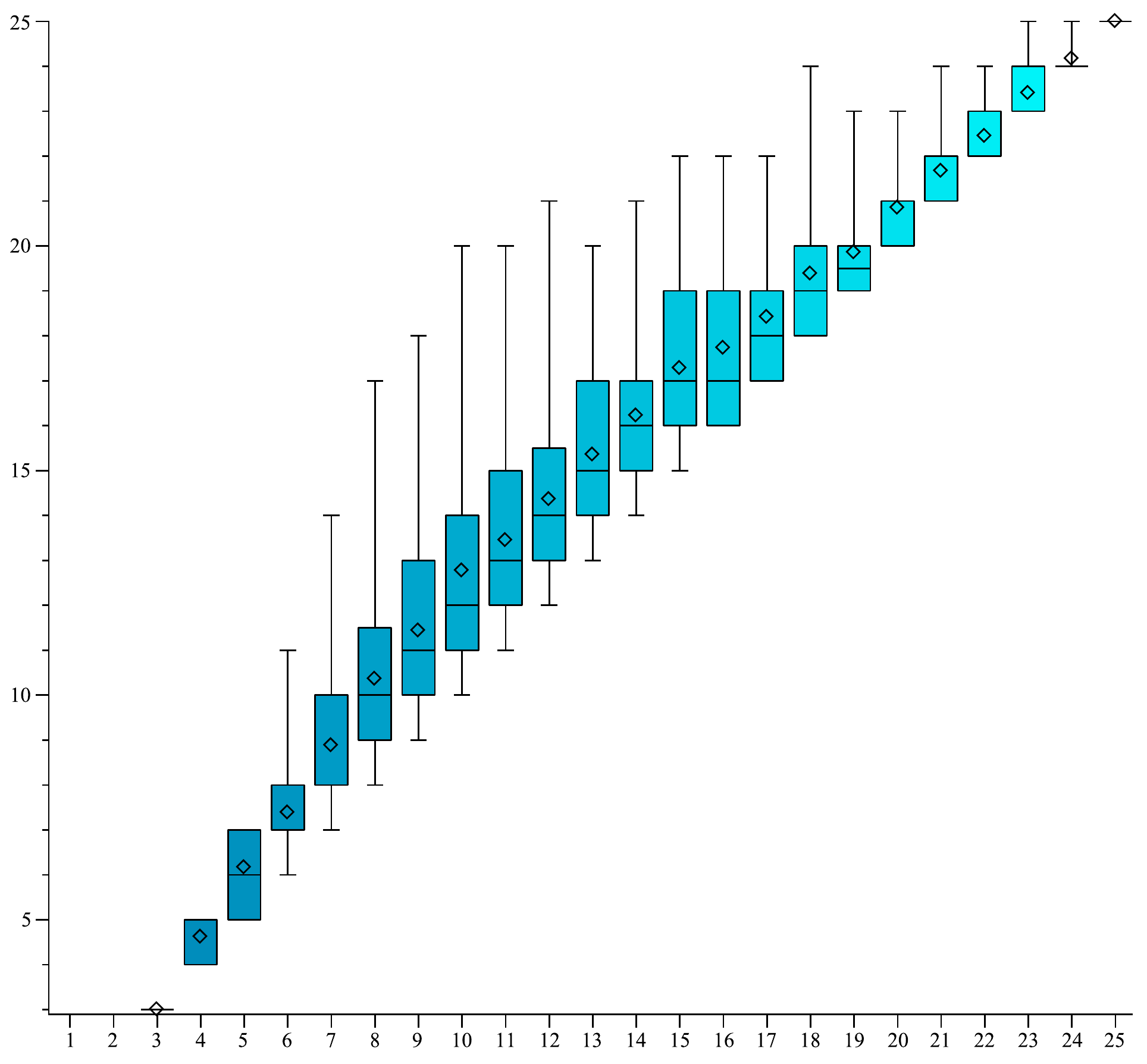}
}

\subfloat[Linear fit for height of resulting chain as a function of height of
original poset: $y=0.9463 x+2.1496$, $R^2=0.9839$]{
\includegraphics[width=3.5in]{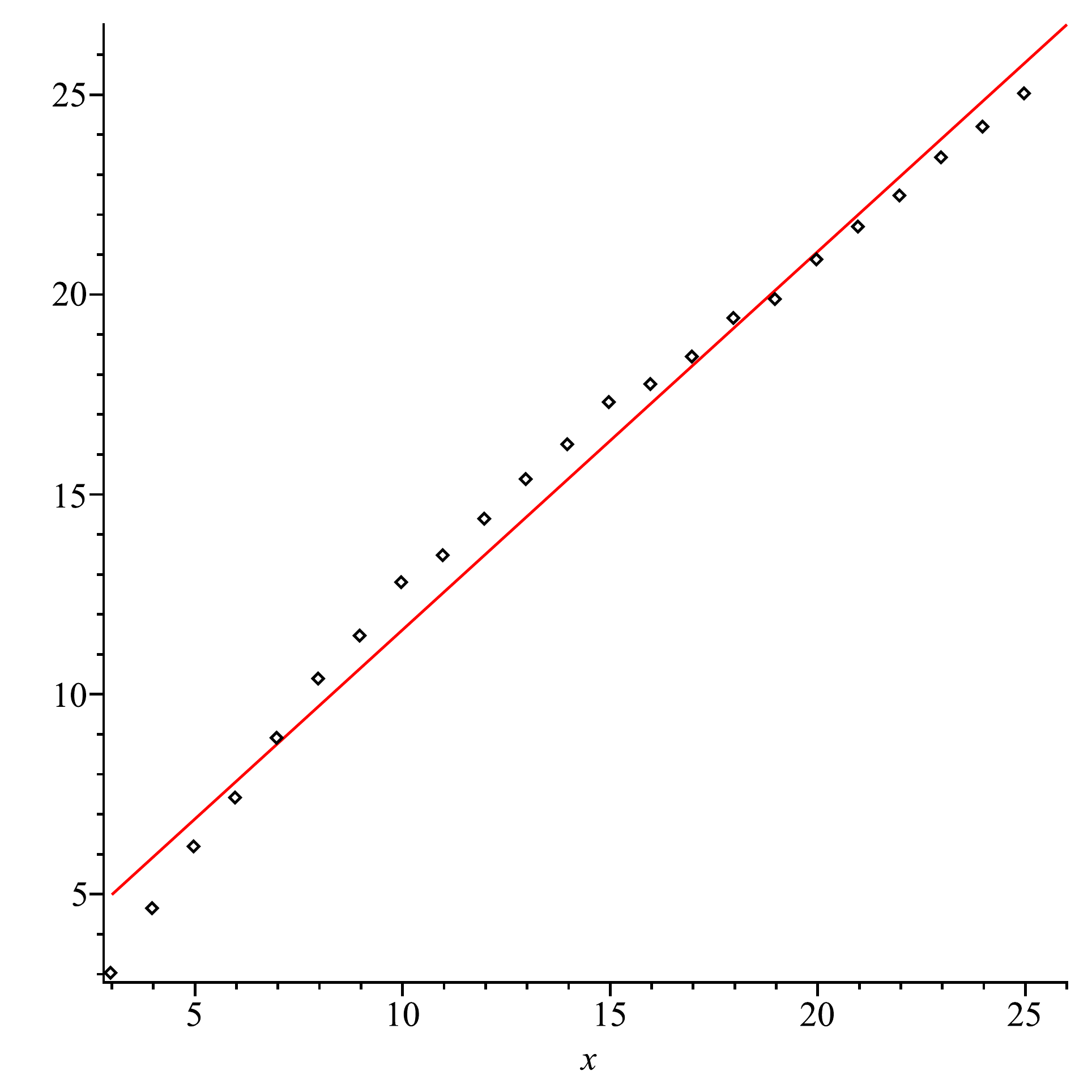}
} \caption{Statistics for the ``height of poset vs. height of resulting
chain" data}\label{HeightStats}
\end{figure}

\section{Experimental Math} \label{ExpAspect}
Throughout Section \ref{structureRpl} we mentioned the use of experiments to
make conjectures or gather data. ``Experimental mathematics'' was once an
oxymoron, however with the advent of computer algebra systems (e.g., Maple
and Mathematica) and more powerful computers, it has gained much more
acceptability. An article in the Notices of the American Mathematical Society
\cite{BaDBoJ11} discusses the many historical and current uses of
experimentation in mathematics. In a list of eight interpretations of
experimental math \cite{BaDBoJ11,BoJBaD08} we use it in our work as a tool
to: (a) gain insight and intuition, (b) test conjectures, and (c) suggest
approaches for formal proof.

\subsection{Generating Small Posets}
As we began to investigate properties of the interval rank poset, $\RplP$, we
found it necessary to construct many examples. Rather than construct
arbitrary posets we constructed the set of all bounded posets of size less
than or equal to 9 using the computer algebra system Maple on a laptop with
an Intel Core i5-2520M processor and 4GB RAM. We did this by first generating
all reflexive $\{0,1\}$-matrices of size between 1 and 7. We considered these
matrices as binary relations, i.e., if the $(i,j)^{th}$ entry is 1 then $x_i
\leq x_j$. Then we checked each matrix for antisymmetry and transitivity: if
$a_{i,j}=a_{j,i}=1$ then $i=j$, and if $a_{i,j}=a_{j,k}=1$ then $a_{i,k}=1$.
We then had to check for isomorphic copies of the same poset, i.e., checking
that permuting the rows and columns of one matrix didn't yield another matrix
in the set. Lastly, we added a top and bottom bound (a row of all 1's and a
column of all 1's). We generated all of the posets of size up to 6 rather
quickly, even checking for isomorphic copies, there are only 318 of size 6.
However, it took many days to generate the 2045 posets of size 7. Bounding
each was then a trivial step. We then had a test set of all 2450 bounded
posets of size between 3 and 9. Of course, any real-world posets would be
considerably larger than size 9, however this test set allowed us to get some
intuition about behavior of the $\Rpl$ operator.

\subsection{Generating Larger Posets}
The test set of 2450 posets was enough for us to get a feeling for the
structure of $\RplP$, and formulate the conjectures that would become
Propositions \ref{heightIncr}, \ref{widthDecr}, \ref{gradedChain}, and
\ref{EventuallyChain}. However, for the work in Section \ref{itrRpl} we
wanted to gather data for larger posets. Because the number of posets of size
$n$, $P_n$, is $2^{n^2/4+o(n^2)}$ \cite{KlDRoB75} it is difficult to get a
representative sample of posets of size $n$.

Our strategy was to generate random posets using two different algorithms,
both found in \cite{BrG1993}. The first is the random graph model. Given some
$n \in \N$ and $0\leq p\leq 1$ generate the Erd\H{o}s-R\'enyi random graph
$G_{n,p}$. This is a graph on $n$ labeled vertices where each edge, $(i,j)$,
is included with probability $p$. We then create a directed graph by
directing each edge from smaller to larger, i.e., if $\{1,2\}$ is an edge
then we direct it from 1 to 2. Notice that this graph must be acyclic since
there can be no decreasing edge. From here we take the transitive closure to
form a random partial order.

The second model is the random $k$-dimensional model. Here, we choose $k$
random linear orders (i.e., $k$ random permutations of $[n]$) and take their
intersection. By the definition of dimension, the resulting partial order has
dimension at most $k$.

\section{Conclusion and Future Work} \label{conclude}

The results reported here point towards a number of continuing efforts.

\begin{description}
\item[Canonical Strong Conjugate:] In this
paper we did not pursue the problem of finding a canonical
    conjugate order to the strong interval order (if one exists).
    While the weak interval order $\le_W$ and
    the subset order $\sub$ stand as conjugates, in \sec{intervals} we
    discussed the availability of conjugates for the strong interval order
    $\le_S$. We have identified a number of possibilities experimentally,
    and would like to cast the question more generally. To do so we need to
    consider the comparability graph of the strong order \cite{gipc}, that
    is, the graph $G=\tup{P,E}$ where $\tup{a,b} \in E \sub P^2$ if $a
    \sim_S b$. The complement of $G$ should be the comparability graph for
    any conjugate order to $\le_S$, if they exist. We also know that the
    complement of the comparability graph of $P$ must be an interval graph,
    that is, the intersection graph of the intervals $a \in P$. So we are
    brought to the question of whether an interval graph can, or must, also
    be a comparability graph. To our knowledge, this is an open question.

\item[Characterizing Iterative Interval Rank Convergence:] In \sec{iterating}
    we analyzed the iterative behavior of the $R^+$ mapping, and determined
    that this converges to a chain. A  natural open question is how
    structural properties of $\po$ can affect the convergence of $\{
    (R^+)^n(\po) \}$ to a chain. Another natural question is whether any of
    the resulting pre-order chains can be equal (even with the same sets of
    elements to be identified), for two different order relations on the set
    $\po$.

It is known that the number of labeled partial orders on $n$ elements is
asymptotically $P_n \sim 2^{\frac{n^2}{4}+\frac{3n}{2}+O(ln(n))}$ \cite{kld},
and the number of labeled preorders is asymptotically $R_n \sim
\frac{n!}{2\ln(2)^{n+1}}$ \cite{kom}. We deal more with unlabeled posets, but
note that if there are two labeled posets, $\po_1 \neq \po_2$ with
$(R^+)^m(\po_1) = (R^+)^\ell(\po_2)$ then there are two unlabeled posets with
the same property. It's clear that $P_n > R_n$ as $n$ goes to $\infty$
($2^{n^2}$ beats $n!$). Therefore, since there are asymptotically more
partial orders than preorders, we must have cases in which $\po_1 \neq \po_2$
but $R^+(\po_1) = R^+(\po_2)$. Additionally, it's not too difficult to show
that for any preorder, $E$, there is a corresponding poset, $\po$, for which
$(R^+)^k(\po) = E$ for $k=1$. So, we ask, what structural properties of
$\po_1 \neq \po_2$ would allow $(R^+)^m(\po_1) = (R^+)^\ell(\po_2)$?

\item[Measures of Gradedness:] Just as we hold that rank in posets is
    naturally and profitably extended to an interval-valued concept, so
    this work suggests that we should consider extending gradedness from a
    qualitative to a quantitative concept. A graded poset is all spindle,
    with all elements being precisely ranked with width 0, and {\it vice
    versa}. Therefore there should be a concept of posets which fail that
    criteria to a greater or lesser extent, that is, being more or less
    graded. In fact, we have sought such measures of gradedness as
    non-decreasing monotonically with iterations of $R^+$. Candidate
    measures we have considered have included the avarege interval rank
    width, the proportion of the spindle to the whole poset, and various
    distributional properties of the set of the lengths of the maximal
    chains. While our efforts have been so far unsuccessful,
    counterexamples were sometimes very difficult to find, and exploring
    the possibilities has been greatly illuminating.

\end{description}


\end{document}